\documentclass[12pt,a4paper]{amsart} 
\usepackage[top=35mm, bottom=35mm, left=30mm, right=30mm]{geometry}
\usepackage{xcolor}
\usepackage[colorlinks=true, citecolor=blue]{hyperref}
\usepackage{verbatim}
\usepackage[nobysame]{amsrefs}

\usepackage{mathptmx}

\usepackage{amsthm,amsfonts}

\usepackage{ulem}

\newtheorem{thm}{Theorem}[section]
\newtheorem{exmp}[thm]{Example}
\newtheorem{lem}[thm]{Lemma}
\newtheorem{prop}[thm]{Proposition}

\newtheorem{que}[thm]{Question}
\theoremstyle{definition}
\newtheorem{remark}[thm]{Remark}
\newtheorem{defn}[thm]{Definition}

\numberwithin{equation}{section}

\DeclareMathOperator{\dimh}{dim_H}

\allowdisplaybreaks[4]

\begin{document}
\title[Hausdorff dimension of   survivor set for beta-transformations ] 
{  The Hausdorff dimension of the  survivor set for beta-transformations  }

\author[R. Kuang]{Rui Kuang}
\address[R. Kuang]{School of Mathematics, South China University of Technology, Guangzhou,
510641, Guangdong, China}
\email{kuangrui@scut.edu.cn}

\author[B. Li]{Bing Li}
\address[B. Li]{School of Mathematics, South China University of Technology, Guangzhou,
510641, Guangdong, China}
\email{scbingli@scut.edu.cn}

	\author[Y. Xiao]{Yuanfen Xiao}
	\address[Y. Xiao]{School of Mathematical Sciences, Xiamen University, Xiamen,
		361005, Fujian, China}
	\email{xyuanfen@mail.ustc.edu.cn}
	
	
	
	\maketitle

\begin{abstract}
Let $ 1<\beta< 2 $, the sequence  $\alpha(\beta)=\alpha(\beta)_1\alpha(\beta)_2\dotsb $ be the quasi-greedy $ \beta $-expansion of $ 1 $,  and $ t\in [0,1) $ be a bifurcation parameter.
The $\beta$-transformation is defined to be $T_{\beta}(x)=\beta x (mod 1) $ for $x\in [0,1)$.
The  Hausdorff dimension of the survivor set  $K(t)=\{x\in [0,1)\colon T_{\beta}^k(x)\not\in (0,t), \forall k\geq0\} $ is equal to $ -\frac{\ln\lambda}{\ln\beta} $ under the condition that
 $ \sum_{i=k}^{\infty}\frac{\alpha(\beta)_i }{\beta^i}\geq t $ for any $ k\geq 1 $,  where  $ \lambda\in (0,1) $ is the smallest positive solution of the equation
 	$\sum_{n=1}^{\infty}(\alpha(\beta)_n-t_n)x^n=1$ with $(t_n) $ being the quasi-greedy $\beta$-expansion of $t$.
  And   the local H\"older exponent of the Hausdorff  dimension function of $K(t) $ is larger than the value of the function itself.
\end{abstract}

\section{Introduction}


Let $f\colon X\to X $ be a self map.
The pair $(X,f) $ becomes an open  dynamical system by designing a subset $Y$ of the phase space $X$ as a hole.
There are two topical contents about the studying of open dynamical systems: the escape rate  and the size of the set of the points that survives from the hole $Y$.
In 1979,  Pianigiani and Yorke \cite{P1979} posed the problem about  the survivor set and how to   characterize  the limiting distributions for a ball
staying in a given measurable set on the table with a small hole.
Then in 1986, Urba\'nski \cite{Urb86,Urb87} introduced  the Hausdorff dimension and the topological entropy  of $K(t) $ for $C^2 $-expanding and orientation preserving map $g\colon S^2\to S^2 $ and  proved that they are  both continuous functions.
He also showed  that the set of bifurcation parameter is homeomorphic to the Cantor set and has zero Lebesgue measure.
After the work of Urba\'nski,
Kalle, Kong, Langeveld and Li\cite{KKLL2020}
considered this situation for the $\beta$-transformation with $1<\beta\leq 2$.
In the following, we use $T_{\beta}$ to represent the $\beta$-transformation, which  is defined to be $T_{\beta}(x)=\beta x (mod 1) $ for $x\in [0,1)$.
In 2020, Agarwal was concerned with the Hausdorff dimension of the survivor set $W_{\beta}(a,b)=\{x\in [0,1)\colon T_{\beta}^k(x)\not \in (a,b) \text{ for any } k\geq 0\} $ when $0\leq a<b\leq 1 $ vary with $\beta\geq 2 $ being an integer.
Other results on the survivor set of $\beta$-transformations  can be found in \cite{Ba2020,MR3005697,MR3233534,MR3345169} among others.

Our motivation is the work of Carminati and Tiozzo in 2017 \cite{CT17}.
They studied the Hausdorff dimension of $K(t)=\{x\in [0,1)\colon T_{\beta}^k(x)\not\in (0,t)\text{ for any } k\geq 0 \}$ and
showed that  the local H\"oldet exponent of the Hausdorff dimension of $K(t)$ is equal to the dimension itself in  the $\beta$-transformation with $\beta\geq 2$ being an integer.
The corresponding $\beta$-shift plays an important role in studying the dynamical properties for $\beta$-transformations.
Let \[
\gamma=\left \{\begin{array}{ll}
	\beta-1 & \text{if }  \beta \text{ is an integer}\\
	\lfloor \beta \rfloor & \text{otherwise}
\end{array} \right. .
\]
In fact, the number $\beta$ being  an integer or not really matters.
When $\beta>1$ is an integer, the $\beta$-transformation is continuous (regarding $[0,1)$ as a circle) and its corresponding $\beta$-shift is just the fullshift over $\beta$ symbols.
But in the case when $\beta>1$ is not an integer,   the $\beta$-transformation is no longer continuous.
This brings a series of problems that we can not adopt the some method used in the case when $\beta$ is an integer.
And even more, its corresponding $\beta$-shift is a proper subsystem of fullshift over $\gamma$ symbols, which means that we can not identify every element in this  $\beta$-shift.

The first main result of this article is about the the Hausdorff dimension of $K(t) $  with $1<\beta<2$.
The method we used here is computing the topological entropy of $K(t)$ and gain the Hausdorff dimension through the relationship between the dimension and the entropy about $K(t)$.
\begin{thm}\label{main1}
 	Let $ 1<\beta< 2 $ and $ t\in \mathcal{U} $.
 	If $ \sum_{i=k}^{\infty}\frac{\alpha(\beta)_i }{\beta^i}\geq t $ for any $ k\geq 1 $, where $ \alpha(\beta)=\alpha(\beta)_1\alpha(\beta)_2\dotsb $ is the quasi-greedy $ \beta $-expansion of $ 1 $, then  $\dimh(K(t)) $ is equal to $ -\frac{\ln\lambda}{\ln\beta} $ with $ \lambda\in (0,1) $ being the smallest positive solution of the equation
 	\[
 	\sum_{n=1}^{\infty}(\alpha(\beta)_n-t_n)z^n=1.
 	\]
 	If the equation has no root in $ (0,1) $, then the Hausdorff dimension of $ K(t) $ is zero.
 \end{thm}
And the second result,   we   study   the local H\"older exponent of the Hausdorff dimension of $K(t)$ with $1<\beta<2$.
\begin{thm}\label{main2}
    Let $ 1<\beta<2 $, $ t\in \mathcal{U} $ and    $ \eta(t)=\dimh(K(t)) $.
 	If $ \sum_{i=k}^{\infty}\frac{\alpha(\beta)_i }{\beta^i}\geq t $ for any $ k\geq 1 $, where $ \alpha(\beta)=\alpha(\beta)_1\alpha(\beta)_2\dotsb $ is the quasi-greedy $ \beta $-expansion of $ 1 $,
  then the local H\"older exponent of $ \eta(t) $ is larger than $ \eta(t)  $.
\end{thm}

The article is organized as follows.
We  introduce some basic knowledge about greedy $\beta$-expansions and quasi-greedy $\beta$-expansions in Section 2.
In Section 3, we discuss the structures of the stable set and the bifurcation set.
In Section 4, we provide the proof of Theorem \ref{main1} and Section 5 is devoting to studying the local H\"older exponent of the Hausdorff dimension of the survivor set by giving the proof of Theorem \ref{main2}.

\section{Preliminaries}

\subsection{Notations on sequences and words}
Let  $ N\geq 1 $ be an integer.
We denote a sequence in $ \{0,1,\dotsc, N\}^{\mathbb{N}} $ by $ (x_i)=x_1x_2\dotsb $ with $x_i\in \{0,1,\dotsc, N\} $ for any $ i\geq 1 $.
Let $ \sigma \colon \{0,1,\dotsc, N\}^{\mathbb{N}} \to \{0,1,\dotsc, N\}^{\mathbb{N}} $  be the shift map such that $ \sigma((x_i))=(x_{i+1}) $.
 A finite string of elements from $ \{0,1,\dotsb,N\} $, is called a word.
%
A word   $ \omega\in \{0,1,\dotsc, N\}^{n} $ is said to be a prefix of a sequence $ (x_i)\in\{0,1,\dotsc, N\}^{\mathbb{N}} $ if $ x_1x_2\dotsb x_n=\omega  $.
Let $ \omega=\omega_1\omega_2\dotsb\omega_n $ be a word with length $ n $.
We use the symbol $ |\omega|=n $ to denote the length of the word $ \omega $.

In order to compare two different sequences, we will use the lexicographical ordering through out this whole article.
A sequence $ (\omega_i)=\omega_1\omega_2\dotsb $ is called lexicographically less than another different sequence $ (\eta_i)=\eta_1\eta_2\dotsb $,  writing as $ (\omega_i)\prec (\eta_i) $, if $ \omega_n<\eta_n $ with $ n=\min\{i\geq 1\colon \omega_i\neq \eta_i \} $.
We write $ (x_i)\preceq  (y_i) $ if $ (x_i)\prec  (y_i) $ or $ (x_i) =(y_i)$.
The lexicographical ordering can also be extended to words.
Let $ u $ and $ v $ be two words in $ \{0,1,\dotsc,\gamma\}^* $.
Then the symbol $ u\prec  v $ means $ u0^{\infty}\prec  v0^{\infty} $.
In the sequel, we will frequently denote a sum by virtue of a sequence or a word.
Fixed a real number $ \beta>1 $.
For a sequence $ (\omega_i)=\omega_1\omega_2\dotsb $ or a word $ u_1u_2\dotsb u_n $, we denote the sum $ \sum_{i=1}^{\infty} \frac{\omega_i}{\beta^i}$ by $ .\omega_1\omega_2\dotsb $ and the sum $ \sum_{i=1}^{n}\frac{u_i}{\beta^i} $  by $ .u_1u_2\dotsb u_n $.

\subsection{$ \beta $-transformation }
Let $ \beta>1 $ be a real number and
\[
\gamma=\left \{\begin{array}{ll}
	\beta-1 & \text{if }  \beta \text{ is an integer}\\
	\lfloor \beta \rfloor & \text{otherwise}
\end{array} \right. .
\]
The (greedy) $ \beta $-transformation is defined to be $ T_{\beta}\colon [0,1)\to [0,1) $ such that $ T_{\beta}(x)=\beta x-\lfloor \beta x\rfloor  $.
For any $ x\in [0,1) $, the (greedy) $\beta $-transformation provides an algorithm to obtain the greedy $ \beta$-expansion of $ x $, which is  denoted by $ b(x,\beta)=(b_i(x,\beta)) $ with
$ b_i(x,\beta)=\lfloor \beta T_{\beta}^{i-1}(x) \rfloor  $ for any $ i\geq 1 $.
It is clear that  $ b(x,\beta) $ is a sequence in $ \{0,1,2,\dotsc,\gamma\}^{\mathbb{N}} $.
Define a map $ \pi\colon \{0,1,\dotsc,\gamma\}^{\mathbb{N}} \to [0,1] $ such that $ \pi_{\beta}((x_i))=\sum _{i=1}^{\infty}\frac{x_i}{\beta^i} $.
It is not hard to see  that
$ x=\sum_{i=1}^{\infty}\frac{b_i(x,\beta) }{\beta^i} $ and in every step from this algorithm we choose the largest integer, which is the origin of the name ``greedy''.
In fact, by the (greedy) $ \beta $-transformation, it follows that $ 1=\frac{\gamma}{\beta}+\sum_{i=2}^{\infty}\frac{ \lfloor\beta T^{i-1}_{\beta}(1) \rfloor}{\beta^i} $, which implies that $ \beta-\gamma= \sum_{i=1}^{\infty}\frac{ \lfloor\beta T^{i}_{\beta}(1) \rfloor}{\beta^i} $.
It is  not hard to show that $ b_i(\beta-\gamma,\beta) $ is equal to $ \frac{ \lfloor\beta T^{i}_{\beta}(1) \rfloor}{\beta^i}  $ for any $ i\geq 1 $.
Therefore, we can write $ b(1,\beta)$ for the sequence $\gamma b(\beta-\gamma,\beta) $.

\begin{lem}\label{lem1}
Let  $ \beta>1 $ and  $ \gamma=\left \{\begin{array}{ll}
\beta-1 & \text{if }  \beta \text{ is an integer}\\
\lfloor \beta \rfloor & \text{otherwise}
\end{array} \right. $.
Each $ x\in (0,1] $ admits exactly one  expansion $ x=\sum_{i=1}^{\infty}\frac{a_i(x)}{\beta^i} $ such that the sequence  $\{a_i(x)\}_{i=1}^{\infty}$ is not eventually zero and $0\leq  a_{i}(x)\leq \gamma $.  	
\end{lem}

\begin{proof}
We first define a special sequence as follow.
Let $b(1,\beta)=b_1(1,\beta)b_2(1,\beta)\dotsb   $ be the greedy $ \beta $-expansion of $ 1 $ with $ b_i(1,\beta)=\lfloor\beta T_{\beta}^{i-1}(1)\rfloor $ for any $ i\geq 1 $.
If the sequence $  b(1,\beta) $ is not eventually zero, then define $ \varepsilon_{\beta}^*(i)=b_i(1,\beta)  $ for any $ i\geq 1 $.
Otherwise, pick $ N=\sup\{i\geq 1\colon  b_i(1,\beta)>0 \} $ and $
N<+\infty $.
 Then we define $ \varepsilon_{\beta}^*(i) $ as follow,
 \[
   \varepsilon_{\beta}^*(i)=\left\{
   \begin{array}{ll}
     b_k(1,\beta) & \text{if } i=k(\mathrm {mod}\, N) \text{ and } k\in \{1,2,\dotsb,N-1\}\\
     b_N(1,\beta)-1 & \text{if } i =0(\mathrm{mod}\, N)
   \end{array}
   \right. .
   \]
  It is not hard to calculate that the sum $ \sum_{i=1}^{\infty}\frac{\varepsilon_{\beta}^*(i)}{\beta^i} $ is equal to $ 1 $.
  Now, for any fixed  $ x\in [0,1) $, if the greedy $ \beta $-expansion of $ x $ is not eventually zero, then $ a_i(x)=b_i(x,\beta)=\lfloor\beta T_{\beta}^{i-1}(x) \rfloor $ for any $ i\geq 1 $.
  For another case, pick $ M=\sup\{i\geq 1\colon \lfloor\beta T_{\beta}^{i-1}(x) \rfloor>0 \} $, then we know that
    $ \lfloor\beta T_{\beta}^{i-1}(x) \rfloor $ is zero for any $ i> M $.
    We choose
  \[
  a_i(x)=\left\{
  \begin{array}{ll}
   \lfloor\beta T_{\beta}^{i-1}(x) \rfloor & i\in\{1,2,\dotsb,M-1\}\\
   \lfloor\beta T_{\beta}^{i-1}(x) \rfloor-1 & i=M\\
   \varepsilon_{\beta}^*(k) & i=M+k \text{ and } k\geq 1
  \end{array}
  \right.,
  \]
   which ends the proof because it is not hard to see that for any $ i\geq 1 $, the value of $ a_i(x) $ is not larger than $ \gamma$.
\end{proof}


In fact, the not eventually zero expansion for $ x $ in the previous lemma is called the quasi-greedy $ \beta $-expansion of $ x $.
We denote the quasi-greedy $ \beta $-expansion of $ x $ by $ \widetilde{b}(x,\beta) $ and write $ \alpha(\beta)=\widetilde{b}(1,\beta) $.
Using these notations, the previous lemma says that if $ b(x,\beta)=b_1b_2\dotsb b_n0^{\infty} $ with $ b_n\geq 1 $, then $ \widetilde{b}(x,\beta) =b_1b_2\dotsb (b_n-1)\alpha(\beta)$.
Fixed $ \beta>1 $,
the following lemma, which can be found in \cite{Par60},
 shows that the quasi-greedy $ \beta $-expansion of $ 1 $, i.e., the sequence $ \alpha(\beta) $, can be used to describe the set of all the greedy $ \beta $-expansions for $ x\in [0,1) $.

 \begin{lem}\label{333}
 	Let $ \beta>1 $ and  $ \gamma=\left \{\begin{array}{ll}
 		\beta-1, & \text{if }  \beta \text{ is an integer}\\
 		\lfloor \beta \rfloor, & \text{otherwise}
 	\end{array} \right. $.
 	Then the sequence $ (\omega_i )\in \{0,1,\dotsc,\gamma \}^{\mathbb{N}} $ is a greedy $ \beta $-expansion of some $ x\in [0,1) $ if and only if   $   \sigma^k((\omega_i )) \prec \alpha(\beta)$ holds for any $ k\geq 0 $.
 \end{lem}

Let  $\Sigma_{\beta}$ be the collection of all the greedy $\beta$-expansion of $x\in [0,1)$, i.e.,
\begin{align*}
\Sigma_{\beta}&=\{ (\omega_i) \in \{0,1,\dotsc, \gamma \} ^{\mathbb{N}} \colon  \text{ there exists } x \in [0,1)
\text{ such that } (\omega_i)=b(x,\beta) \} \\
&=\{(\omega_i) \in \{0,1,\dotsc, \gamma \} ^{\mathbb{N}} \colon 0^{\infty}\preceq  \sigma^k((\omega_i))\prec \alpha(\beta) \text{ for any } k\geq 0     \}.
\end{align*}
Let $ \widetilde{\Sigma}_{\beta} $ be the collection of all quasi-greedy $ \beta $-expansions of $ x\in (0,1] $.
From Lemma~\ref{lem1} and Lemma~\ref{333},  it follows that
$$ \widetilde{\Sigma}_{\beta}=\{(\omega_i)\in \{0,1,\dotsc, \gamma \} ^{\mathbb{N}}  \colon 0^{\infty}\prec  \sigma^k((\omega_i))\preceq \alpha(\beta)  ,\forall k\ge0\} .$$
The following result can be found in \cite{KKLL2020}, \cite{Par60}. 

\begin{lem}
	Let $ \beta>1 $.
	The map $ x\mapsto b(x,\beta) $ is a strictly increasing bijection from $ [0,1) $ to $ \Sigma_{\beta} $. 
	The map $ x\mapsto \widetilde{b}(x,\beta) $ is a strictly increasing bijection from $ (0,1] $ to $ \widetilde{\Sigma}_{\beta} $.
\end{lem}

\subsection{The Hausdorff dimension }

Fix a real number $ s>0 $ and a subset $ E\subset X $.
Let $ \delta>0 $. Then the $ s $-dimensional Hausdorff measure of $ E $ is given by
\[
\mathcal{H}^s(E)=\lim_{\delta\to 0}\inf\{\sum_{i=1}^{\infty}\text{diam}(E_i)^s\colon E\subset \bigcup_{i=1}^{\infty}E_i \text{ and } \text{diam}(E_i)<\delta \  \text{for all } i\geq 1   \},
 \]
where $ \text{diam}(A) $ means the diameter of the set $ A $.
We now  define the Hausdorff dimension of the set $ E $ is defined  as follows,
\[
\text{dim}_H(E)=\Big\{
	\begin{array}{ll}
		\inf\{s>0\colon \mathcal{H}^s(E)<+\infty \},& \text{ if }\{s>0\colon \mathcal{H}^s(E)<+\infty \}\neq \emptyset; \\
		+\infty ,& \text{ otherwise. }
	\end{array}
\Big.
\]

\section{The bifurcation set }

Let $ \mathcal{P}([0,1))=\{A\colon A\subset [0,1)\} $. We define a set-valued function $ K \colon [0,1]\to  \mathcal{P}([0,1)) $  as the following for $0<t<1$,
$$ K(t)=\{x\in [0,1)\colon T_{\beta}^k(x)\not \in (0,t)\text{ for any } k\geq 0\} .$$
It is nature to set $ K(0) =[0,1)$, $ K(1)=\{0\} $. Clearly, $K(t')\subseteq K(t)$ if $t'\geq t$.

A parameter $ t\in [0,1] $ is said to be a bifurcation parameter if $t=0,1$ or $0<t<1$ and for any $ \delta>0 $, there exists some $ t'\in(t-\delta,t+\delta) $ such that $ K(t)\not =K(t') $.
Set
\[
 \mathcal{U}=\{t\in [0,1)\colon t \text{ is a    bifurcation parameter} \}.
\]
It is easy to see that $K(\cdot)$ is locally constant at $t$ for $t\not\in \mathcal{U}$. Such a $t\not\in \mathcal{U}$ is called stable. In \cite{KKLL2020} the authors consider the right set valued bifurcation set $E_{\beta}=\{t\in[0,1):K(t')\neq K(t) \text{ for any } t'>t \}.$
The following lemma shows that the two notions of bifurcation sets are equivalent.

\begin{lem}\label{K}
	 Let $ t\in [0,1)$. Then
\[t \text{ satisfying that } t\in K(t)\Leftrightarrow t\in E_{\beta}\Leftrightarrow t\in \mathcal{U}.\]
\end{lem}

\begin{proof}
	 It is obviously that $t \text{ satisfying that } t\in K(t)\Rightarrow t\in E_{\beta}$ since the number  $ t $ is  in $ K(t) $ but not in $ K(t') $ for any $t'>t$.
It also follows that $t\in E_{\beta}\Rightarrow t\in \mathcal{U}$  by the definitions of the two sets.

Now we show $t\in \mathcal{U}\Rightarrow t \text{ satisfying that } t\in K(t).$
	Let $ t $ be a bifurcation parameter. For any $ \varepsilon>0 $, there exists some $ t'\in[t-\varepsilon,t)\cup(t,t+\varepsilon] $ such that $ K(t')\neq K(t) $.
	 For $\varepsilon_n=\frac{1}{n} $,
	 when $ t_n $ is in $ (t-\varepsilon_n,t)\cup(t,t+\varepsilon_n) $,  there exists $ x_n $  in $ \big( K(t_n)\setminus K(t)\big)\cup \big(K(t)\setminus K(t_n) \big)  $.
	 There will be two kinds of different cases.

	 If there exist  infinite $ n\in\mathbb{N} $ such that  $ t_n  \in (t,t+\varepsilon_n] $, then there exists some $ l_n\geq 0 $ such that $ T_{\beta}^{l_n}(x_n)\in [t,t_n) $ while $ T_{\beta}^l(x_n)\in [t,1] $ for any $ l\geq 0 $.
	 Set $ z_n=T_{\beta}^{l_n}(x_n
	 ) $.
	 Observe that $ z=\lim_{n\to\infty} z_n=t $ and $T_{\beta}^m(t)=T_{\beta}^m(z)=T_{\beta}^m( \lim_{n\to\infty} z_n  )\in [t,1]$ for any $ m\geq 0 $ by the right continuity of $ \beta $-transformation, which shows that $ t $ is in $ K(t) $.	
	
	 Otherwise, for the case when there are infinite $ n\in\mathbb{N} $ such that  $ t_n  \in   (t-\varepsilon_n,t) $, then there exists some $ k_n\geq 0 $ such that $ T_{\beta}^{k_n}(x_n)\in [t_n,t) $ while $ T_{\beta}^k(x_n)\in [t_n,1] $ for any $ k\geq 0 $.
	 Set $ y_n=T_{\beta}^{k_n}(x_n) $. Then $ T_{\beta}^m(y_n)=T_{\beta}^{m+k_n}(x_n)\geq t_n $ and $ \lim_{n\to\infty} y_n=t $.
	 Assuming that $ t\not\in K(t) $, there exists $ m> 0 $ such that $ T_{\beta}^m(t)\in (0,t) $.
	 Fixing the number $ m$, we have that  the map $ T_{\beta}^m $ is continuous at the point $ t $.
	So $T_{\beta}^m(t)=\lim_{n\to\infty}T_{\beta}^m( y_n  )\geq t $,
	which is contract to $ T_{\beta}^m(t)\in (0,t) $.
\end{proof}

Next we study the structure of the set of stable points $[0,1]\setminus\mathcal{U}$. We call $ r\in (0,1) $ a $ \beta $-rational number if $ r $ admits a finite greedy $ \beta $-expansion.
Denote the set of $ \beta $-rational numbers by $ \mathbb{Q}_{(\beta)} $.
Let $ r $ be a  $ \beta $-rational number and $  b(r,\beta)= r_1r_2\dotsb r_m0^\infty=r_1r_2\dotsb r_m00\dotsb $ be its finite greedy $ \beta $-expansion with $ r_m $ being the last non-zero element.
We define an interval $ I_r=(.r_1r_2\dotsb r_m,.\overline{r_1r_2\dotsb r_m }) $,  where the notation $ .\overline{r_1r_2\dotsb b_m } $ means the number $ \overline{r}=.r_1r_2\dotsb r_mr_1r_2\dotsb r_mr_1r_2\dotsb r_m\dotsb $.
It needs to clarify that the sequence $  \overline{r_1r_2\dotsb r_m }=r_1r_2\dotsb r_mr_1r_2\dotsb r_m\dotsb $ may not be  in $ \Sigma_{\beta} $.
For example,   when $ \beta=\frac{3+\sqrt{13}}{2} $, the quasi-greedy $ \beta $-expansion of $ 1 $ is $ 3030303030\dotsb $.
According to Lemma~\ref{333},  we know that the sequence
  $ 230^{\infty} $ can be the greedy $ \beta $-expansion of some real number $ r $ in $ [0,1) $, while $ \overline{r}=.\overline{23} $ is not in $ \Sigma_{\beta} $ because  $\alpha(\beta)\prec \sigma(\overline{r})=323232\dotsb  $.

If $\overline{r_1r_2\dotsb r_m } \not\in \Sigma_{\beta}$ then $j:=\inf\{k\ge0:\sigma^k \overline{r_1r_2\dotsb r_m }\succeq  \alpha(\beta)\}\in\{0,1,2,\dotsb,m-1\}$.
Since it always hold that $\gamma w_1w_2\dotsb\succeq  w_1w_2\dotsb$ we have $j\ge1$ and $r_j<\gamma$ if $r_1<\gamma$.
We define a new sequence $\widetilde{r}$ as $\widetilde{r}:=\alpha(\beta)$ if $r_1=\gamma$; $\widetilde{r}:=r_1r_2\dotsb r_{j-1}(r_j+1)0^\infty$ if $r_1<\gamma$.
If $\widetilde{r}=r_1r_2\dotsb r_{j-1}(r_j+1)0^\infty\not\in \Sigma_{\beta}$ we repeat the above method to get a new sequence $\widetilde{r}^{(1)}=r_1r_2\dotsb r_{j'-1}(r_{j'}+1)0^\infty$
where $j'=\inf\{k\ge0:\sigma^k \widetilde{r}\succeq  \alpha(\beta)\}\in\{1,2,\dotsb,j-1\}$. The process will be stopped after finite repeating.
We denote the final sequence by $\widetilde{r}'$.
That is $\overline{r_1r_2\dotsb r_m }\rightarrow \widetilde{r}\rightarrow\widetilde{r}^{(1)}\rightarrow\widetilde{r}^{(2)}\rightarrow\dotsb\rightarrow\widetilde{r}^{(k)}=\widetilde{r}'\in \Sigma_{\beta}.$
For instance, let $1<\beta<2$ and $0<r<1$ satisfy that $\alpha(\beta)=1101$, $b(r,\beta)=0110011$, then $\overline{0110011}\rightarrow 011010^\infty\rightarrow 10^\infty.$

\begin{defn}
For $r\in \mathbb{Q}_{(\beta)} $, $  b(r,\beta)= r_1r_2\dotsb r_m000\dotsb $ with $r_m\not=0$. We define the interval $ I_r^*:=(r,r^*)$ where
$ {r^*}=\left \{\begin{array}{ll}
 		.\overline{r_1r_2\dotsb r_m }, & \text{if }    \overline{r_1r_2\dotsb r_m } \in \Sigma_{\beta}\\
 		.\widetilde{r}', & \text{otherwise }
 	\end{array} \right. .$
\end{defn}

\begin{lem}\label{K1}
 Let $ t\in[0,1)\setminus\mathcal{U} $. Then there exists some $ r\in \mathbb{Q}_{(\beta)} $ such that $ t\in I_r^* $. Moreover the two endpoints of $ I_r^* $ are both in $ \mathcal{U} $.
\end{lem}

\begin{proof}	
	Pick $ t\in [0,1)\setminus \mathcal{U} $. 	
	Let $ b(t,\beta)=(t_i) $ be the greedy $ \beta $-expansion of $ t $.
	As $ t\not \in K(t) $, there exists some integer $ k_0\geq 1 $ such that $ T_{\beta}^{k_0}(t)\in (0,t) $.
	Pick $ k $ be the smallest integer such that $  T_{\beta}^{k}(t)\in (0,t)  $.
	Let $ r =.t_1t_2\dotsb t_k $. Clearly, $ t_k\not=0 $ by the smallness of $ k $, and $r<t$.
	We are going to show that $ t<r^* $.
    Denote $W_n=t_{k(n-1)+1}\dotsb t_{kn}$ be the $n$-th block word in $(t_i) $ for $n=1,2,3,\dotsb$. Then $  T_{\beta}^{k}(t)<t  $ implies that
   $W_2W_3W_4\dotsb\prec W_1W_2W_3\cdots$. So there is $l\ge1$ such that $W_1=W_2=\dotsb=W_l$ and $W_{l+1}\prec W_l$. Thus, we get $t<r^*$ by the following
   \[(t_i)=W_1W_2W_3\dotsb\prec (W_1)^\infty=\overline{t_1t_2\dotsb t_k}\prec \widetilde{r}'.\]

	Now, we are going  to show that $ r $ is in $ \mathcal{U} $.
	We show the {\bf claim:} $ T_{\beta}^{h}(r) > r$ for any $ h\in \{1,2,\dotsc,k-1\} $.
	By contradiction, if there exists some $ h\in \{1,2,\dotsc,k-1\}  $ such that $ .t_{h+1}t_{h+2}\dotsb t_k<.t_1t_2\dotsb t_k $, then we choose $ l=\min\{ 1\leq i\leq k-h\colon t_{h+i}\neq t_i \} $.
	As $ t_{h+l}<t_l $, it leads to $ .t_{h+1}t_{h+2}\dotsb t_kt_{k+1}\dotsb<.t_1t_2\dotsb $.
	This is contract to the definition of $ k $. 	The claim is valid. Thus $ r\in K(r) $ and $ r\in\mathcal{U} $.
	
    It remains to show that   $r^*\in\mathcal{U} $.	If $r^*=1$ then $r^*\in \mathcal{U}$ by the definition. For $r^*=.\overline{t_1t_2\dotsb t_k }$ it is easy to show that
    $\sigma^h\overline{t_1t_2\dotsb t_k }\succ  \overline{t_1t_2\dotsb t_k }$ for any $ h\in \{1,2,\dotsc,k-1\}  $ by the above claim. So $r^*\in K(r^*)$ and $r^*\in \mathcal{U}.$
	For the case of $\overline{t_1t_2\dotsb t_k }\not\in\Sigma_{\beta}$, without loss of generality we assume that $t_1=0$ and
    \[r^*=.t_1t_2\dotsb t_{j-1}(t_j+1)\in \Sigma_{\beta},\ j=\inf\{i\ge1:\sigma^i \overline{t_1t_2\dotsb t_k }\succeq  \alpha(\beta)\}.\]
    By the above claim $ t_{h+1}t_{h+2}\dotsb t_k\succ  t_1t_2\dotsb t_k $ for any $ h\in\{1,2,\dotsc,j-1\} $. It must hold that
    $ t_{h+1}t_{h+2}\dotsb t_j\succeq  t_1t_2\dotsb t_{j-h} $. So $ t_{h+1}t_{h+2}\dotsb t_{j-1}(t_j+1)\succ  t_1t_2\dotsb t_{j-h} $ and
    \[ t_{h+1}t_{h+2}\dotsb t_{j-1}(t_j+1)\succ  t_1t_2\dotsb t_{j-h}\dotsb t_{j-1}(t_j+1) .\]
     So $ T_{\beta}^{h}(r^*) > r^*$ for any $ h\in \{1,2,\dotsc,j-1\} $. We obtain that $ r^*\in K(r^*) $ and $ r^*\in \mathcal{U} $.
\end{proof}

\begin{prop}\label{stable-set}
	The connected components of  $ [0,1)\setminus\mathcal{U} $ are parameterized by $ \beta $-rational numbers, which means that the following result is true,
	\[
	[0,1)\setminus \mathcal{U} =\bigcup_{r\in \mathbb{Q}_{(\beta)}}I_{r}^*.
	\]
\end{prop}

\begin{proof}
	Fix $ r_1r_2\dotsb r_m00\dotsb\in \mathbb{Q}_{(\beta)} $ with $ r_m\neq 0 $ be the finite greedy $ \beta $-expansion of $ r $.
	For any $ t\in I_r^* $,  according to Lemma~\ref{K} and Lemma~\ref{K1} we just need to show that $ t\not \in K(t) $.

	Let $ b(t,\beta)=(t_i) $ be the greedy $ \beta $-expansion of $ t $.
    If $r_1=1$ then $K(t)=\{0\}$ for all $t\in(r,1)$. For the case of $r^*=.\overline{r_1r_2\dotsb r_m }$,
	since $ t\in I_r^* $,  $ (t_i) $ must have the form $(t_i)=W_1W_2\dotsb W_kW_{k+1}\dotsb $ where $W_1=W_2=\dotsb=W_k=r_1r_2\dotsb r_m$ and $W_{k+1}\prec r_1r_2\dotsb r_m$ for some $k\geq1$.
	So $ 0<T_{\beta}^{m}(t)<t $ and   $ t \not\in K(t) $.
    For the case of $r_1=0$ and $\overline{r_1r_2\dotsb r_m }\not\in\Sigma_{\beta}$, without loss of generality we assume that
    \[r^*=.r_1r_2\dotsb r_{j-1}(r_j+1)\in \Sigma_{\beta},\ j=\inf\{i\ge1:\sigma^i \overline{r_1r_2\dotsb r_m }\succeq  \alpha(\beta)\}.\]
    Clearly, $(t_i)\not=\overline{r_1r_2\dotsb r_m }.$
   If $r_1r_2\dotsb r_m0^\infty\prec (t_i)\prec \overline{r_1r_2\dotsb r_m }$  we have get $ 0<T_{\beta}^{m}(t)<t $  in the second case.
   If $\overline{r_1r_2\dotsb r_m }\prec (t_i)\prec r_1r_2\dotsb r_{j-1}(r_j+1)0^\infty$   there is $h\ge j+1 $ such that $t_i=r_{(i\ \mathrm{mod}\, m)},1\le i \le h-1$ and $t_h>r_{(h\ \mathrm{mod}\, m)}$.
   Therefor, \[\sigma^j(t_i)=t_{j+1}\dotsb t_h\dotsb \succ   \sigma^j(\overline{r_1r_2\dotsb r_m })\succeq  \alpha(\beta).\]
   It is a contradiction with $(t_i)\in \Sigma_{\beta}$. It always holds that  $r_1r_2\dotsb r_m0^\infty\prec (t_i)\prec \overline{r_1r_2\dotsb r_m }$.  So we get $ 0<T_{\beta}^{m}(t)<t $ and $t\not\in K(t). $
\end{proof}

\begin{remark}\label {stable-set1}
	 We say a word $ \omega $ is Lyndon, if for each decomposition $ \omega=AB $ into two non-empty words, it always has $ \omega\prec  B $.
	 And a real number $ r\in \mathbb{Q}_{(\beta)} $ is called Lyndon, if its greedy $ \beta $-expansion $ r_1r_2\dotsb r_k $ is a Lyndon word with $ r_k $ is last non-zero number.
	 In fact, for every $ r\in \mathbb{Q}_{(\beta)} $, the   number  $ r $ is in $ \mathcal{U} $ if and only if $ r$ is a Lyndon number.
	 Therefore,  if we denote the set of all the Lyndon number by $ \mathbb{Q}_{(\beta,Lyn)} $, then  it implies that
	 \[
	 [0,1)\setminus \mathcal{U} =\bigcup_{r\in \mathbb{Q}_{(\beta)}}I_{r} ^*=\bigsqcup_{r\in \mathbb{Q}_{(\beta,Lyn)}}I_{r}^*,
	 \]
where ``$\bigsqcup$" means a disjoint union.
\end{remark}

Let $ S=s_1s_2\dotsb s_m $ be a word with length $ m $ and $ x\in [0,1) $.
We use the symbol $ S\cdot x $ to express the real number  $ S\cdot x=\frac{s_1}{\beta}+\frac{s_2}{\beta^2}+\dotsb+\frac{s_m}{\beta^m}+\frac{x}{\beta^m} $.
Set
\[
{\Sigma}_{\beta,n}=\{(a_i)\in \{0,1,\dotsc, \gamma \}^n\colon  a_1a_2\dotsb a_n 0^\infty \in {{\Sigma}}_{\beta}\}.
\]
Let $ t\in [0,1) $ and $ b(t,\beta)= (t_i ) $ be its greedy $ \beta $-expansion.
For any $ k\geq 1 $  and any $ s\in \{0,1,\dotsc,\gamma\} $, we define a word $ S_{k,s}(t)=t_1t_2\dotsb t_{k-1}s $.
Set $ \Omega_{\beta,k}(t)= \{S_{k,s}(t)\colon s>t_k \} $  and  $ \Omega_{\beta}(t)=\bigcup_{k\geq 1} \Omega_{\beta,k}(t) $, we are going to characterize the structure of  $ K(t) $ through $ \Omega_{\beta}(t) $.

\begin{prop}\label{K(t)}
	 If $ t\in \mathcal{U} $, then
	 we have the following result
	 \[
	 K(t)\subset \{t\}\cup \bigcup_{k=1}^{\infty} \Bigg(  \bigsqcup_{S\in \Omega_{\beta,k}(t)\cap{\Sigma}_{\beta,k}} S\cdot K(t)\Bigg).
 	 \]
\end{prop}

 \begin{proof}
 	By Lemma \ref{K}, the number $ t $ must be in $ K(t) $ because of  $ t\in \mathcal{U} $. Let $ b(t,\beta)= (t_i) $ be the greedy $ \beta $-expansion of $ t $ and $ b(x,\beta)= (x_i) $ be the greedy $ \beta $-expansion of $ x $.
 	 Set $ l=\min\{n\geq 1\colon t_n\neq x_n \} $.
 	 Since $ x $ is larger than $ t $, we must have $ x_l>t_l $ and $ x_i=t_i $ for any $ 1\leq i\leq l-1 $.
 	 Now, set $ S=t_1t_2\dotsb t_{l-1}x_l $.
 	 It is clear that  the word $S $ is in $ \Omega_{\beta,l}(t)\cap {\Sigma}_{\beta,l} $. 	
 	 Observe that the set $ K(t) $ is $ T_{\beta} $-invariant, i.e. $ T_{\beta}(K(t))\subset K(t) $ and $ \pi_{\beta}\circ\sigma=T_{\beta}\circ\pi_{\beta}   $, the number $ .x_{l+1}x_{l+2}\dotsb  $ is still in $ K(t) $ and hence, the number $x=S\cdot (.x_{l+1}x_{l+2}\dotsb)$ is in $S\cdot K(t)$.
\end{proof}

\section{The Haudorff dimension of the survivor set}
    Raith considered the topological pressure, the topological entropy and the Hausdorff dimension of the survivor set for piecewise monotone expanding maps on $ [0,1] $.
    He proved that the Hausdorff dimension of the survivor set is the unique zero of the pressure function, see\cite[Lemma 3]{Rai94} and \cite[Lemma 4]{Rai92}.
    In our setting, this means the following result is true.
\begin{lem}\label{dimension-entropy}
	 Let $ \beta>1 $, then
	 \[
	 \dimh(K(t))\cdot\ln\beta=h_{top}(K(t),T_{\beta}).
	 \]
\end{lem}
	In order to obtain the Hausdorff dimension of $ K(t) $, we study its topological entropy.
We define a metric $ d $ on $ \{0,1,\dotsc,\gamma\}^{\infty} $ as
\[
d((x_i),(y_i) )= \sum_{i=1}^{\infty}\frac{|x_i-y_i|}{\beta^i}.
\]
Recalled that we have defined a map $\pi\colon \{0,1,\dotsc,\gamma\}^{\infty}\to [0,1]$ such that $\pi((x_i))=\sum_{i=1}^{\infty}\frac{x_i}{\beta^i} $.
Let $\widetilde{\pi}_{\beta}=\pi|_{\widetilde{\Sigma}_{\beta}}\colon \widetilde{\Sigma}_{\beta}\to (0,1] $.
Note that the map $ \widetilde{\pi}_{\beta} $ satisfies $ \widetilde{\pi}_{\beta}\circ \sigma=T_{\beta}\circ \widetilde{\pi}_{\beta} $ and is a homeomorphism, it follows that $ h_{top}(Y,T_{\beta})=h_{top}(\widetilde{\pi}_{\beta}^{-1}(Y),\sigma) $ holds for any subset $Y\subset (0,1] $.
In particular, we have $ h_{top}(K(t),T_{\beta})=h_{top}(\widetilde{\pi}_{\beta}^{-1}(K(t)),\sigma)$.

In the rest of this article, we consider the special case of $1<\beta\leq 2$.
 	Let $(\{0,1\}^{\mathbb{N}},\sigma)$ be the full shift over two symbols.
 	For $\mathbf{t}=(t_i)$ and  $\mathbf{u}=(u_i)\in \{0,1\}^{\mathbb{N}}$, we define a subset  $\mathcal{K}(\mathbf{t},\mathbf{u})$ of $\{0,1\}^{\mathbb{N}}$ as follow,
 	\[
 	\mathcal{K}(\mathbf{t},\mathbf{u})=\{\mathbf{x}=(x_i)\in \{0,1\}^{\mathbb{N}}\colon (t_i)\preceq   \sigma^n( (x_i)) \preceq  (u_i)  \text{ for all } n\geq 0\}.
 	\]
 	Often, for convenience, we will write it  as
 	\[
 	\mathcal{K}(\mathbf{t},\mathbf{u})=\{\mathbf{x}\in \{0,1\}^{\mathbb{N}}\colon \sigma^n (\mathbf{x}) \in [\mathbf{t},\mathbf{u}]  \text{ for all } n\geq 0\}.
 	\]
 	 It is clear that $ \mathcal{K}(\mathbf{t},\alpha(\beta))\bigtriangleup\widetilde{\pi}_{\beta}^{-1}(K(t)) $ is a countable set, where $A\bigtriangleup B=(A\setminus B)\cup (B\setminus A)$.
     So $ h_{top}(\mathcal{K}(\mathbf{t},\alpha(\beta)),\sigma)=h_{top}(\widetilde{\pi}_{\beta}^{-1}(K(t),\sigma) $ by the variational principle of entropy.
 	 In \cite[Proposition 2.6]{KKLL2020} the authors also give a proof via a particular computation.

     Thus, we get the following equation and we just need to consider the topological entropy of $\mathcal{K}(\mathbf{t},\mathbf{u})$.
\[\dimh(K(t))=\frac{h_{top}(K(t),T_{\beta})}{\ln\beta}=\frac{h_{top}(\mathcal{K}(\mathbf{t},\alpha(\beta)),\sigma)}{\ln\beta}.\]

 \begin{remark} It is not hard to obtain the following results.
 	\begin{enumerate}
 		\item $\mathcal{K}(\mathbf{t},\mathbf{u})$ is  closed and strongly $ \sigma $-invariant, i.e., $ \sigma(\mathcal{K}(\mathbf{t},\mathbf{u}))= \mathcal{K}(\mathbf{t},\mathbf{u})  $.
 		
 		\item $\mathcal{K}(\mathbf{t},\mathbf{u})\subset\{0^\infty,1^\infty\}$ if $t_1=1$ or $u_1=0$.
 		
 		\item $\mathcal{K}(\mathbf{t} ,\mathbf{u})$  consists of finite periodic points if $\sum_{n=1}^\infty \frac{u_n-t_n}{2^n}< \frac{1}{2}$.
 		
 	\end{enumerate}
 \end{remark}

 \begin{thm}\label{entropy}
 	Assuming that $\mathbf{t}=(t_i)$ and $\mathbf{u}=(u_i)\in  \{0,1\}^{\mathbb{N}}$  are two elements of $ \mathcal{K}=\mathcal{K}(\mathbf{t},\mathbf{u})$. That is,
 	 \[
 	 t_1t_2t_3\cdots \preceq   \sigma^n (t_1t_2t_3\cdots) \preceq  u_1u_2u_3\cdots
 	 \] and
 	 \[
 	 t_1t_2t_3\cdots \preceq   \sigma^n (u_1u_2u_3\cdots) \preceq  u_1u_2u_3\cdots
 	 \]
 	 hold for any $n\geq 0$.
 	 Then the topological entropy $h_{top}(\mathcal{K},\sigma)=-\ln\lambda$,
 	where $\lambda\in (0,1)$ is the smallest positive solution of the equation
   \begin{equation}
   		\sum_{n=1}^{\infty}(u_n-t_n){z^{n}}=1. \label{eq}
   \end{equation}	
 	If the equation has no root in $(0,1)$, then $h_{top}(\mathcal{K},\sigma)=0$.
 \end{thm}

 It is clearly that if the equation \eqref{eq} has a root, then the root must be in  $[\frac{1}{2},1)$.
 Furthermore,  if  the root of the equation \eqref{eq} exists, then it  must be in $[\frac{1}{\beta},1)$ where $\beta\in (1,2]$ satisfies that
\[ 1=\sum_{n=1}^{\infty}\frac{u_n}{\beta^n}.\]
To prove this theorem we first consider several invariant subsets of $(\{0,1\}^{\mathbb{N}},\sigma)$.
 For $\mathbf{a}=(a_i)$, $ \mathbf{b}=(b_i)\in \{0,1\}^{\mathbb{N}}$ with $a_1=0$  and  $b_1=1$,   define
 \begin{align*}
 E(\mathbf{a},\mathbf{b})&=\{\mathbf{x}=(x_i)\in \{0,1\}^{\mathbb{N}}\colon \sigma^n (\mathbf{x})\preceq  \mathbf{a}   \text{ or }  \mathbf{b}\preceq  \sigma^n (\mathbf{x})  \text{ for all } n\geq 0\}\\
&=\{(x_i)\in \{0,1\}^{\mathbb{N}}\colon \sigma^n (x_1x_2x_3\cdots)\not\in(a_1a_2\cdots,b_1b_2\cdots)  \text{ for all } n\geq 0\}
\end{align*}
 and
 \begin{align*}
  F(\mathbf{a},\mathbf{b})=\{\mathbf{x}\in \{0,1\}^{\mathbb{N}}\colon &\sigma(\mathbf{b})\preceq \sigma^n (\mathbf{x})\preceq  \mathbf{a}   \text{ or }\\
 &  \mathbf{b}\preceq  \sigma^n (\mathbf{x})\preceq \sigma (\mathbf{a})  \text{ for all } n\geq 0\}.
 \end{align*}

 \begin{prop} \label{subshift}
 	Let $\mathbf{a}=(a_i)$  and  $ \mathbf{b}=(b_i) $ be two sequences in  $  \{0,1\}^{\mathbb{N}}$ with $a_1=0$  and $b_1=1$.
 	 Then
 	\begin{enumerate}
 		\item $F(\mathbf{a},\mathbf{b})= \mathcal{K}(\sigma (\mathbf{b}),\sigma(\mathbf{a})).$

 		\item $E(\mathbf{a},\mathbf{b})\cap [\sigma(\mathbf{b}),\sigma(\mathbf{a})]=F(\mathbf{a},\mathbf{b})= \mathcal{K}(\sigma (\mathbf{b}),\sigma(\mathbf{a})).$
 		
 		\item $h_{top}(\mathcal{K}(\sigma(\mathbf{b}),\sigma(\mathbf{a})),\sigma)=h_{top}(E(\mathbf{a},\mathbf{b}),\sigma).$
 	\end{enumerate}
 \end{prop}

 \begin{proof}
 	 In the rest of the proof, we write $\mathcal{K}=\mathcal{K}(\sigma (\mathbf{b}),\sigma(\mathbf{a})),F=F(\mathbf{a},\mathbf{b}),E=E(\mathbf{a},\mathbf{b})$ for convenience.
 	
 	(1) We first show  that  the set $ \mathcal{K} $  coincides with $F$.
 	Let $\mathbf{x}\in F$ and  $n\geq 0$.
 	Then according to the definition of $ F $,
 		we have $\sigma(\mathbf{b})\preceq \sigma^n (\mathbf{x})\preceq  \mathbf{a}$ or $b\preceq  \sigma^n (\mathbf{x})\preceq \sigma (\mathbf{a})$.
 	Notice that  $\mathbf{a}\preceq  \sigma (\mathbf{a})$ and $\sigma(\mathbf{b})\preceq  \mathbf{b}$ by the condition $a_1=0$  and  $b_1=N$, it turns out that
 	$\sigma(\mathbf{b})\preceq \sigma^n (\mathbf{x})\preceq  \mathbf{a}\preceq  \sigma (\mathbf{a})$  or  $\sigma(\mathbf{b})\preceq  \mathbf{b}\preceq  \sigma^n (\mathbf{x})\preceq \sigma (\mathbf{a})$ holds.
     Eventually, it always holds that
 	$\sigma(\mathbf{b})\preceq  \sigma^n (\mathbf{x})\preceq \sigma (\mathbf{a})$ and thus $\mathbf{x}\in \mathcal{K}$.
 	
 	Pick $\mathbf{x}=(x_i)\in \mathcal{K}$ and  $n\geq 0$.
 	Then for any $ m\geq 0 $, the sequence $ \mathbf{x} $ satisfies that  $\sigma(\mathbf{b})\preceq  \sigma^{m} (\mathbf{x})\preceq \sigma (\mathbf{a})$  by the definition of the set $ \mathcal{K} $.
 	If $ x_{n+1}=0 $, then from $ x_{n+2}x_{n+3}\dotsb \preceq  a_2a_3\dotsb  $, we have $x_{n+1}x_{n+2}\dotsb \preceq  a_1a_2\dotsb  $.
 	Therefore, in this case, the sequence $\mathbf{x}  $ is contained in $ F $.
 	If  $ x_{n+1}=1$, then from
 	$ b_2b_3\dotsb \preceq  x_{n+2}x_{n+3}\dotsb  $, we obtain that $ b_1b_2\dotsb \preceq  x_{n+1}x_{n+2}\dotsb $.
 	This also implies that $ \mathbf{x} $ is an element of the set $ F $.

 	\medskip

    (2)
 	It is clearly that $\mathcal{K}\subset [\sigma(\mathbf{b}),\sigma(\mathbf{a})]$ and $F\subset E$.
 	So $F=F\cap \mathcal{K}\subset E\cap[\sigma(\mathbf{b}),\sigma(\mathbf{a})]$.
 	For the other direction, it is sufficient to show that  if $\mathbf{x}\in E\cap [\sigma(\mathbf{b}),\sigma(\mathbf{a})]$, then
 	$\sigma(\mathbf{x})\in [\sigma(\mathbf{b}),\sigma(\mathbf{a})]$ because of $ \sigma(E)\subset E $.
 	If $x_1=0$, then from  $\mathbf{x}\preceq  \mathbf{a}$ and $\sigma(\mathbf{x})\preceq  \sigma (\mathbf{a})$,
 	we get that $\sigma (\mathbf{b})\preceq  \mathbf{x}\preceq  \sigma (\mathbf{x})\preceq  \sigma (\mathbf{a})$ and thus $\sigma(\mathbf{x})\in [\sigma(\mathbf{b}),\sigma(\mathbf{a})].$
 	If $x_1= 1$, then from $\mathbf{b}\preceq  \mathbf{x}$ and $\sigma(\mathbf{b})\preceq  \sigma (\mathbf{x})$,
 	it turns out that  $\sigma (\mathbf{b})\preceq  \sigma(\mathbf{x})\preceq  \mathbf{x}\preceq  \sigma (\mathbf{a})$ and thus $\sigma(\mathbf{x})\in [\sigma(\mathbf{b}),\sigma(\mathbf{a})].$
 	
 	\medskip
 	
 	(3) It is easy to see that $E=\{0^\infty, 1^\infty\}$ if $\mathbf{a}=00a_3a_4\cdots$ or $\mathbf{b}=11b_3b_4\cdots$ and therefore $h_{top}(\mathcal{K},\sigma)=h_{top}(E,\sigma)=0.$
     It remains to consider the case    when $\mathbf{a}=01a_3a_4\dotsb$ and  $\mathbf{b}=10b_3b_4\dotsb$.
 	We claim that  for each  $ \mathbf{x}\in E\setminus \{0^\infty, 1^\infty\}$, there exists $m\geq 0$ such that $\sigma^m(\mathbf{x})\in \mathcal{K}$.
 	If it is not true, then there exists some  sequence $ \mathbf{x}\in E\setminus \{0^\infty, 1^\infty\} $ such that $ \sigma^n(\mathbf{x})\not \in \mathcal{K}=E\cap [\sigma(\mathbf{b}),\sigma(\mathbf{a}) ] $ for any $ n\geq 0 $.
 	Especially, $ \mathbf{x} $ is not in $ \mathcal{K}  $.
 	And furthermore, $ \mathbf{x} $ is  not in $ [\sigma(\mathbf{b}),\sigma(\mathbf{a})] $  either.
 	As $ \sigma(E)\subset E $ and $ \sigma(\mathbf{x})\not \in \mathcal{K}=E\cap [\sigma(\mathbf{b}),\sigma(\mathbf{a})] $, it turns out that $ \sigma(\mathbf{x})\in E\setminus \{0^\infty, 1^\infty\}  $  and
 	 $\sigma(\mathbf{x})\not\in [\sigma(\mathbf{b}),\sigma(\mathbf{a})]$.
 	 Finally, we can assume that $ \sigma^n(\mathbf{x})\not \in [\sigma(\mathbf{b}),\sigma(\mathbf{a})] $
 	  for all   $n\geq0$.
 	
 	If $x_1=0$, then $\mathbf{x}\preceq  \mathbf{a}  $ because of $ \mathbf{x}\in E $.
 	It follows that  $ \sigma(\mathbf{x})\preceq  \sigma(\mathbf{a}) $.
 	Recalled that we have assumed that $ \sigma^n(\mathbf{x})\not \in [\sigma(\mathbf{b}),\sigma(\mathbf{a})] $ for any $ n\geq 0 $, we have $ \sigma(\mathbf{x}) \prec  \sigma (\mathbf{b}) $.
 	This implies that $x_2=x_3=x_4=\cdots=0$.

 	If $x_1=1$, then $ \mathbf{b}\preceq  \mathbf{x}  $ because of $ \mathbf{x}\in E $.
 	It follows that $ \sigma(\mathbf{b})\preceq  \sigma(\mathbf{x})  $ and $ \sigma(\mathbf{a})\prec  \sigma (\mathbf{x}) $ by the same hypothesis.
 	Finally,   we get $x_2=x_3=x_4=\cdots=1$.
 	
 	So far, the proof of the claim is finished.
 	 And for all $   \mathbf{x}\in E\setminus \{0^\infty, 1^\infty\}$, the $\omega-$limit set $\omega(\mathbf{x},\sigma)$  is contained in  $ \mathcal{K}$.
 	$$h_{top}(\mathcal{K},\sigma)\leq h_{top}(E,\sigma)=\sup_{\mathbf{x}\in E}h_{top}(\omega(\mathbf{x},\sigma),\sigma)\leq h_{top}(\mathcal{K},\sigma).$$
 	The ``$=$" in the above equation is a general property of topological entropy \cite{Denker1976}.
 \end{proof}

 The next result is from the Lemma 3 in \cite{BSV14}.
 \begin{prop} \label{lemma3}
 	Let $\mathbf{a}=(a_i)$,  $b=(b_i)\in \{0,1\}^{\mathbb{N}}$ with $a_1=0$,  $b_1=1$ and $\mathbf{a},\mathbf{b}\in E(\mathbf{a},\mathbf{b})$.
 	The entropy $h_{top}(E(\mathbf{a},\mathbf{b}),\sigma)>0$.
 	Then $h_{top}(E(\mathbf{a},\mathbf{b}),\sigma)=-\ln\lambda$,
 	where $\lambda\in (0,1)$ is the smallest positive root of the power series
 	\[\mathcal{K}(z)\colon=\sum_{n=1}^{\infty}(b_n-a_n){z^{n-1}}.\]
 \end{prop}

 \noindent {\itshape The proof of Theorem \ref{entropy}}. Let $\mathbf{a}=0{\mathbf{u}}=0u_1u_2\cdots, \mathbf{b}=1{\mathbf {t}}=1t_1t_2\cdots$.
 Observe the proof of the third statement of  Proposition \ref{subshift}, it only needs to consider that case when $ u_1=1 $ and $ t_1=0 $.
 We claim that both  $ \mathbf{a} $ and $ \mathbf{b} $ are actually contained in $ E(\mathbf{a},\mathbf{b})  $.

 If $  \mathbf{a} \not \in  E(\mathbf{a},\mathbf{b}) $, then there exists some $ n\geq 0 $ such that
 \begin{equation}
 0u_1u_2u_3\dotsb \prec  u_{n}u_{n+1}u_{n+2}\dotsb \prec 1t_1t_2t_3\dotsb.\label{3.10.1}
 \end{equation}
As the sequence $ \mathbf{u} $ satisfies $ \sigma^n(u_1u_2u_3\dotsb)\preceq  u_1u_2u_3\dotsb$, it follows that $ u_{n}=1 $.
Then, by \eqref{3.10.1}, we have $ u_{n+1}u_{n+2}\dotsb \prec  t_1t_2\dotsb $, which is contract to the premise that $ \mathbf{u}\in \mathcal{K}(\mathbf{t},\mathbf{u} ) $.
Thus, the sequence $ \mathbf{a} $ is in fact contained in $ E(\mathbf{a},\mathbf{b} ) $.

Next, we are going to show that the sequence $ \mathbf{b} $ is also contained in $ E(\mathbf{a},\mathbf{b} ) $.
 If $  \mathbf{b} \not \in  E(\mathbf{a},\mathbf{b}) $, then there exists some $ m\geq 0 $ such that
\begin{equation}
	0u_1u_2u_3\dotsb \prec  t_{m}t_{m+1}t_{m+2}\dotsb \prec 1t_1t_2t_3\dotsb.\label{3.10.2}
\end{equation}
Because $t_1t_2\dotsb \preceq  t_{m+1}t_{m+2}\dotsb $, the number $ t_m $ must be equal to $ 0 $.
By \eqref{3.10.2}, we have $ u_1u_2u_3\dotsb\prec  t_{m+1}t_{m+2}t_{m+3}\dotsb $, which is contract to the condition that $ \mathbf{t}\in \mathcal{K}(\mathbf{t},\mathbf{u} ) $.
Therefore, we finish the proof of our claim.
 It follows that
 $h_{top}(\mathcal{K},\sigma)=h_{top}(E(\mathbf{a},\mathbf{b}),\sigma)$ from the Proposition \ref{subshift}.
 It is obviously that the smallest positive root of
 $\mathcal{K}(z)$ is just the smallest positive solution of the equation $\sum_{n=1}^{\infty}(u_n-t_n){z^{n}}=1$.
 So we complete the proof.

We refer the reader to \cite[Theorem 2.3]{KL07} for the following useful lemma concerned about the property the quasi-greedy $ \beta $-expansion of $ 1 $.
 \begin{lem}\label{alpha}
 	For any $ \beta>1 $,  as the quasi-greedy $ \beta $-expansion of $ 1 $, the sequence $ \alpha(\beta) $	satisfies that $ \sigma^n(\alpha(\beta))\preceq  \alpha(\beta) $ for any $ n\geq 0 $.
 \end{lem}

 \begin{proof}[Proof of Theorem~\ref{main1}]
  Let $ \mathbf{t} $ be the quasi-greedy $ \beta $-expansion of $ t $.
  Pick $ \mathbf{u}=\alpha(\beta) $.
  Then, from Lemma~\ref{alpha}, we know that both $ \mathbf{t} $ and $ \mathbf{u} $ satisfy the condition of Theorem~\ref{entropy}.
  From Lemma~\ref{dimension-entropy} and the statement concerned about the relation between $ K(t) $ and $ \mathcal{K}(\mathbf{t},\alpha(\beta)) $,
  we can conclude that $\dimh(K(t)) $ is equal to $ -\frac{\ln\lambda }{\ln\beta}$ with $ \lambda\in (0,1) $ being the smallest positive solution of the equation
  \[
  \sum_{n=1}^{\infty}(\alpha(\beta)_n-t_n)z^n=1.
  \]
  If the equation has no root in $ (0,1) $, then the Hausdorff dimension of $ K(t) $ is zero.
   \end{proof}

  Now we turn to discuss the calculation of the Hausdorff dimension of $K(t)$ for any $\beta\in (1,2]$ and $t\in [0,1)$. It is sufficient to study the general case of $\mathbf{t}\not\in\mathcal{K}(\mathbf{t},\mathbf{u})$ or $\mathbf{u}\not\in\mathcal{K}(\mathbf{t},\mathbf{u})$. The next proposition shows that the general case can be transformed to the spacial one in Theorem \ref{entropy}.

  \begin{prop}
  If $\mathcal{K}(\mathbf{t},\mathbf{u})\not=\emptyset$
 we denote $\mathbf{a}=\min\mathcal{K}(\mathbf{t},\mathbf{u}),\, \mathbf{b}=\max\mathcal{K}(\mathbf{t},\mathbf{u})$. Then
 \[\mathbf{a},\mathbf{b}\in \mathcal{K}(\mathbf{a},\mathbf{b})=\mathcal{K}(\mathbf{t},\mathbf{u}).\]
  \end{prop}
  \begin{proof}
  Since $\mathbf{t}\preceq\mathbf{a}\preceq\mathbf{b}\preceq\mathbf{u}$ we just prove $\mathcal{K}(\mathbf{a},\mathbf{b})\supseteq\mathcal{K}(\mathbf{t},\mathbf{u}).$ For any $\mathbf{x}\in \mathcal{K}(\mathbf{t},\mathbf{u})$
  and any $n\ge0$, $\sigma^n\mathbf{x}\in\mathcal{K}(\mathbf{t},\mathbf{u})$. So $\mathbf{a}\preceq\sigma^n\mathbf{x}\preceq\mathbf{b}$ and $\mathbf{x}\in \mathcal{K}(\mathbf{a},\mathbf{b})$.
  \end{proof}

  Let's find the minimal value and maximal value of $\mathcal{K}(\mathbf{t},\mathbf{u})(\not=\emptyset)$  in several steps.

  {\bf Step 1.} Let  \[ \mathbf{a^{(1)}}=\left \{\begin{array}{ll}
 		\mathbf{t}, & \text{if }  \sigma^k\mathbf{t}\succeq \mathbf{t},\forall k\ge 0;\\
 		(t_1t_2\dotsb t_{n_1})^\infty, & \text{if }  n_1=\inf\{k\ge1:\sigma^k\mathbf{t}\prec \mathbf{t}\}<\infty,
 	\end{array} \right. \]
\[ \quad \ \ \mathbf{b^{(1)}}=\left \{\begin{array}{ll}
 		\mathbf{u}, & \text{if }  \sigma^k\mathbf{u}\preceq \mathbf{u},\forall k\ge 0;\\
 		(u_1u_2\dotsb u_{m_1})^\infty, & \text{if }  m_1=\inf\{k\ge1:\sigma^k\mathbf{u}\succ \mathbf{u}\}<\infty.
 	\end{array} \right. \]
 Then $\sigma^k\mathbf{a^{(1)}}\succeq \mathbf{a^{(1)}},\sigma^k\mathbf{b^{(1)}}\preceq \mathbf{b^{(1)}},\forall k\ge1$ and  $\mathcal{K}(\mathbf{t},\mathbf{u})=\mathcal{K}(\mathbf{a^{(1)}},\mathbf{b^{(1)}}).$

  \medskip

  {\bf Step 2.} Let
    \[ \mathbf{a^{(2)}}:=\tau(\mathbf{a^{(1)}})=\left \{\begin{array}{ll}
 		\mathbf{a^{(1)}}, & \text{if }  \sigma^k\mathbf{a^{(1)}}\preceq \mathbf{u},\forall k\ge 0;\\
 		(t_1t_2\dotsb (t_{n_2}+1))^\infty, & \text{if }  n_2=\inf\{k\ge1:\sigma^k\mathbf{a^{(1)}}\succ \mathbf{u}\}<\infty.
 	\end{array} \right. \]
   Then $\sigma^k\mathbf{a^{(2)}}\succeq \mathbf{a^{(2)}}\succeq\mathbf{a^{(1)}},\forall k\ge1$ and  $\mathcal{K}(\mathbf{a^{(1)}},\mathbf{b^{(1)}})=\mathcal{K}(\mathbf{a^{(2)}},\mathbf{b^{(1)}}).$

  We repeat the action $\tau$ on $\mathbf{a^{(2)}}$ to get $\mathbf{a^{(3)}}$ if $n_3=\inf\{k\ge1:\sigma^k\mathbf{a^{(2)}}\succ \mathbf{u}\}<\infty$.
  The process will be stopped in finite times repeating since $n_1>n_2>n_3>\dotsb.$  We will get a sequence $\mathbf{a^{(1)}}\prec \mathbf{a^{(2)}}\prec\dotsb\prec \mathbf{a^{(h)}}=\tau(\mathbf{a^{(h)}})$ in the process
  and $\mathbf{a}:=\mathbf{a^{(h)}}=\min\mathcal{K}(\mathbf{t},\mathbf{u}).$

   \medskip

   {\bf Step 3.} Let
   \[ \mathbf{b^{(2)}}:=\theta(\mathbf{b^{(1)}})=\left \{\begin{array}{ll}
 		\mathbf{b^{(1)}}, & \text{if }  \sigma^k\mathbf{b^{(1)}}\succeq \mathbf{t},\forall k\ge 0;\\
 		(u_1u_2\dotsb (u_{m_2}-1))^\infty, & \text{if }  m_2=\inf\{k\ge1:\sigma^k\mathbf{b^{(1)}}\prec \mathbf{t}\}<\infty.
	\end{array} \right. \]
 Then $\sigma^k\mathbf{b^{(2)}}\preceq \mathbf{b^{(2)}}\preceq\mathbf{b^{(1)}},\forall k\ge1$ and  $\mathcal{K}(\mathbf{a^{(1)}},\mathbf{b^{(1)}})=\mathcal{K}(\mathbf{a^{(1)}},\mathbf{b^{(2)}}).$

  Same as step 2 we get a sequence $\mathbf{b^{(1)}}\succ \mathbf{b^{(2)}}\succ\dotsb\succ \mathbf{b^{(j)}}=\theta(\mathbf{b^{(j)}})=:\mathbf{b}=\max\mathcal{K}(\mathbf{t},\mathbf{u}).$

\begin{remark} It is easy to see that
\begin{enumerate}
  \item The order of Step 2 and Step 3 can be exchanged.
  \item Replacing $\mathbf{u}$ with $\mathbf{b^{(1)}}$ in Step 2 can reduce the number of repetitions to get $\mathbf{a}$.
  \item Replacing $\mathbf{t}$ with $\mathbf{a}$ in Step 3 can reduce the number of repetitions to get $\mathbf{b}$.
\end{enumerate}
\end{remark}

\begin{exmp}
For $\mathcal{K}(\mathbf{t},\mathbf{u})=\mathcal{K}(0101101110^\infty,111010010001^\infty)$,  we have
\[\mathbf{t}=0101101110^\infty\rightarrow (010110111)^\infty\rightarrow (010111)^\infty\rightarrow (011)^\infty=\mathbf{a},\]
and
\[\mathbf{u}=111010010001^\infty\rightarrow(11101001000)^\infty\rightarrow (110)^\infty=\mathbf{b}.\]
Thus, $\mathcal{K}(\mathbf{t},\mathbf{u})=\mathcal{K}(\mathbf{a},\mathbf{b})=\{(011)^\infty,(101)^\infty,(110)^\infty\}.$
\end{exmp}


Since Proposition \ref{lemma3}  can not be applied to the  sequences in $\{0,1,2\}^{\mathbb{N}} $ and other spaces  with more than 2 symbols (The space $\{0,1,2\}^{\mathbb{N}} $ has 3 symbols.), we have the following question.
\begin{que}
Let $\beta>2$ be a real number. Can we  evaluate the topological entropy and the Hausdorff dimension of $ K(t)$ with $t\in \mathcal{U}$ for $\beta$-transformations?
\end{que}

\section{The local H\"older exponent}

A function $ f\colon I\to \mathbb{R} $ defined on an interval $ I $ is said to be H\"older continuous of exponent  $ \alpha $, if there exists a constant $ C>0 $ such that $ |f(x)-f(y)|\leq C|x-y|^{\alpha} $ for any $ x,y\in I $.
Fix $ t\in I $, the local H\"older exponent of $ f $ at the point $ t $ is given by
\[
\alpha(f,t)=\liminf_{t'\to t}\frac{\ln|f(t')-f(t)|   }{\ln|t'-t| }.
\]
Let $\eta(t)=\dimh(K(t))$.
This section is devoted to studying  the local H\"older exponent of    $\eta(t)$.
Let $ u $ and  $ t $  be two real numbers in $ [0,1] $.
Let $\widetilde{b}(u,\beta)  = (u_i) $  and $ \widetilde{b}(t,\beta) = (t_i) $ be the quasi-greedy $ \beta $-expansion of $ u $ and $ t $ respectively.
Pick
 \[
 m(t,u) =\sup\{k\geq 1\colon t_i=u_i \text{ for any } i\in \{1,2,\dotsc,k \}  \}
 \]
 as the largest length of the same prefix that  $\widetilde{b}(u,\beta)$ and  $\widetilde{b}(t,\beta)$ share with.

\begin{lem}\label{yj}
	 For each $ t_0>0 $, there exists  a constant $ C_1>0 $ such that
	 \[
	 C_1\beta^{-m(t,u) }\leq |t-u|\leq \beta^{-m(t,u)}
	  \]
	  holds for any $ t,u\in \mathcal{U}\cap [t_0,1] $.
\end{lem}

\begin{proof}
	Set $ m=m(t,u) $.
	 Let $ t=q_1q_2\dotsb q_m  t_{m+1}t_{m+2}\dotsb $  and $ u=q_1q_2\dotsb q_m u_{m+1}u_{m+2}\dotsb $ be the quasi-greedy $ \beta $-expansion of $ t $ and $ u $.
	 As
	 $ |t-u|=|\frac{t_{m+1}-u_{m+1}}{\beta^{m+1}}+\frac{t_{m+2}-u_{m+2}}{\beta^{m+2}}+\dotsb  | $.
	 It is clear that $ |t-u| $ is not more than $ \beta^{-m(t,u) } $.
	 Note that
	  \[
	 T_{\beta}^{-1}([t_0,1])\subset \bigsqcup_{k=0}^{\gamma-1}\bigg[\frac{t_0+k}{\beta}, \frac{1+k}{\beta} \bigg]\bigcup\bigg[\frac{\gamma}{\beta},1\bigg),
	 \]
	and $ u_{m+1}\neq t_{m+1} $,  the point $ T_{\beta}^m(t) $ and $ T_{\beta}^m(u) $ can not be in the same interval.
	Therefore, $ |T_{\beta}^m(t)-T_{\beta}^m(u) |\geq \frac{t_0}{\beta} $ and $ |t-u|=\beta^{-m}|T_{\beta}^m(t)-T_{\beta}^m(u)|\geq C_1\beta^{-m}  $ with $ C_1=\frac{t_0}{\beta} $.
	\end{proof}

\begin{proof}[Proof of Theorem \ref{main2}]
Let $ \widetilde{b}(t,\beta)=(t_i) $ be the quasi-greedy  $ \beta $-expansion of $t$  and $ \alpha(\beta)=\alpha(\beta)_i $ be the quasi-greedy  $ \beta $-expansion of $ 1 $.

   In the case when $ t $ is in $ (\frac{\gamma}{\beta},1 ]\cap \mathcal{U} $,
 the number $ t_1$  must be equal to $ \gamma $ and  $\gamma=\alpha(\beta)_1=\gamma $.
For any $ x\in [t,1) $, set $ \widetilde{b}(x,\beta)=(x_i) $ be the the quasi-greedy  $ \beta $-expansion of $ x $. Note that $ x_1=\gamma $ since $ x\geq t $. As $ x $ is less than $ 1 $, there must exists some positive integer $ k>1 $ such that $ x_k<\alpha(\beta)_k\leq \gamma $.
So, $ T_{\beta}^{k-1}(x)<t $    implies that $ K(t)=\{0\} $ and $ \eta(t)=\alpha(\eta,t)=0 $.

For another case when $ t  $ is in $ [0,\frac{\gamma}{\beta}]\cap \mathcal{U} $.
 Set
\[
P_t(x)=\sum_{n=1}^{\infty}(\alpha(\beta)_n-t_n){x^{n}}.
\]
By calculation,  $ P_t(0)=0 $, $ \lim_{x\to 1^-}P_t(x)\geq 1 $ and $ 0< P_t'(x)\leq \frac{2\gamma}{(1-x)^2} $ for any $ x\in (0,1) $.
To sum up, there must exists only one solution $ \lambda(t) $ in $ (0,1] $ and $ \lambda(t) $ is always a simple root of  the equation $ P_t(x)=1 $.
Note that the topological entropy of $K(t)$ can not be more than $\ln \beta$ that is the topological entropy of the whole system, then $  \lambda(t)  $ is larger than $\frac{1}{\beta}$ by Theorem~\ref{main2} and Lemma~\ref{dimension-entropy}.
The function $ \eta(t)=-\frac{\ln \lambda(t)}{\ln \beta} $
	has the property that there exists two    constants $ C_1=\frac{1}{\ln\beta} >0 $ and $ C_2=\frac{\beta}{\ln\beta}>0 $ such that
	\[
	C_1|\lambda(v)-\lambda(u)|\leq  |\eta(v)-\eta(u)|\leq C_2|\lambda(v)-\lambda(u)|
	\]
	for any $\frac{1}{\beta} \leq \lambda(u)<\lambda(v)\leq 1$,
	because
      \[
		 \ln \lambda(v)-\ln \lambda(u)=\ln\frac{\lambda(v)}{\lambda(u)} <\frac{\lambda(v)}{\lambda(u)}-1\leq  \beta|\lambda(v)-\lambda(u)| \]
		 and
		 \[
		  \ln \lambda(v)-\ln \lambda(u)=\ln\frac{\lambda(v)}{\lambda(u)}>1-\frac{\lambda(u)}{\lambda(v)}\geq|\lambda(v)-\lambda(u)|.
		 \]

    When $ t\in [0,\frac{\gamma}{\beta}] \cap \mathcal{U}  $, we pick up two numbers $ v  $, $ u\in [0,\frac{\gamma}{\beta}] \cap \mathcal{U} $ sufficiently close to $ t $.
    Let $ \widetilde{b}( v,\beta)=(v_i) $ be its quasi-greedy  $ \beta $-expansion and $ \widetilde{b}( u,\beta)=(u_i) $ be the quasi-greedy  $ \beta $-expansion of $ u $.
    Set $ P_1(x)=P_{v}(x) $, $ P_2(x)=P_{u}(x) $,$ \lambda_1=\lambda(v) $ and $ \lambda_2=\lambda(u) $.
    Assume that there exists a positive real number $ \delta=\delta(\lambda_1,\lambda_2)>0 $ such that $ \lambda_1\leq 1-\delta $ and $ \lambda_2\leq 1-\delta $.
    note that $ P_2(\lambda_2)=P_1(\lambda_1)=1 $ and  applying Lagrage's theorem, there exists some $ \xi\in[\lambda_1,\lambda_2] $ such that
    \[
    P_1(\lambda_1)-P_2(\lambda_1)=P_2(\lambda_2)-P_2(\lambda_1)=P_2'(\xi)(\lambda_2-\lambda_1)
    \]
	holds.
	We can also write $ P_1(\lambda_1)-P_2(\lambda_1) $ as the power series as follows,
	\[
	P_1(\lambda_1)-P_2(\lambda_1)=\lambda_1^{m+1}R(v,u),
	\]
	where $ R(v,u)=(u_{m+1}-v_{m+1})+\sum_{j=m+2}^{+\infty} (u_j-v_j )\lambda_1^{j-m-1}  $ and $ m=m(v,u) $.
	By comparing the two previous equations we get
	\[
      |\lambda_1-\lambda_2|=\lambda_1^{m+1}\frac{|R(v,u) | }{ P_2'(\xi)}\leq \lambda_1^{m+1}\frac{\gamma}{1-\lambda_1}.
	\]
	We can obtain $m\leq \frac{\ln |v-u| }{-\ln \beta }  $ by
	using the upper bound for $ |v-u| $ in Lemma~\ref{yj}.
	Then, for each $ t\in  [0,\frac{\gamma}{\beta}]\cap\mathcal{U} $, there exists a constant $ \frac{\gamma}{\delta}>0 $ such that
	\[
	|\lambda_1-\lambda_2|\leq \lambda_1^{m+1}\frac{\gamma}{1-\lambda_1}=e^{\ln|v-u|\cdot \frac{-\ln \lambda_1}{\ln \beta }   }\frac{\gamma}{\delta}=|v-u|^{\frac{-\ln \lambda_1}{\ln \beta }}\frac{\gamma}{\delta}
	=|v-u|^{\eta(t)}\frac{\gamma}{\delta}
	\]
	holds for each $ v,u \in \mathcal{U}$ sufficiently close to $ t $.
	Since $ \lambda(t) $ is constant on the complement of $ \mathcal{U} $, the above supper bound actually works for any $ v,u $ close to $ t $.
	Hence,
  \begin{align*}
      \alpha(\eta,t)=&\liminf_{t'\to t}\frac{\ln|\eta(t')-\eta(t)| }{\ln|t'-t|}\\
      \geq & \frac{\ln|\lambda(t')-\lambda(t)|+\ln C_1 }{\ln|t'-t|}\\
      = & \frac{\ln|\lambda(t')-\lambda(t)|+\ln (\frac{1}{\ln\beta} )  }{\ln|t'-t|}\\
      >&\frac{\ln|\lambda(t')-\lambda(t)|-\ln \frac{\gamma}{\delta} }{\ln|t'-t|}
      \geq\eta(t)
  \end{align*}
	holds for any $ t\in\mathcal{U} $ because $\gamma=1$ when $1<\beta<2$.
\end{proof}

The following example is interesting and satisfies that $\alpha(\eta,t)>\eta(t) $.
\begin{exmp}
Let $1<\beta<2$ satisfying $b(1,\beta)=11101$ and $t\in(0,1)$ with $b(t,\beta)=01$. To calculate $\eta(t)=\dim_H(K(t))$ we consider the following sub-shift
\[\mathcal{K}=\mathcal{K}(010^\infty,111010^\infty)=\mathcal{K}((01)^\infty,(110)^\infty)=\mathcal{F}\{00,111\},\]
where $\mathcal{F}\{00,111\}$ denotes the finite type sub-shift forbidding the words $00$ and $111$.
There are two ways to calculate the entropy of the sub-shift. Using Theorem \ref{entropy} we get that $h_{top}(\mathcal{K},\sigma)=-\ln\lambda$, where $\lambda\in (0,1)$ is the solution of $x+x^5=1$.
As well as, applying the method for finite type sub-shifts the number $\lambda$ should be the solution of $x^2+x^3=1$.
 Assuming  $\lambda^2+\lambda^3=1$, we have \[\lambda+\lambda^5=\lambda+\lambda^2(1-\lambda^2)=\lambda+\lambda^2-\lambda^4=\lambda+\lambda^2-\lambda(1-\lambda^2)=\lambda^2+\lambda^3=1.\]
 Note that the solution in $(0,1)$ exists uniquely for each equations, the results obtained by the two methods are consistent.

Clearly, $t=.01\in\mathcal{U}$. Let's consider $r:=.00111$ and computer $\eta(r)=\dim_H(K(r))$. 
We have the following translation
\[ \mathcal{K}(001110^\infty,111010^\infty)=\mathcal{K}((00111)^\infty,(11100)^\infty).\]
We get the related equation $\frac{x+x^2-x^4-x^5}{1-x^5}=1$ and the simplification $(x-1)(x^2+x^3-1)=0$. So the smallest root in $(0,1)$ is just $\lambda$
and $\eta(r)=\frac{-\ln\lambda}{\ln\beta}$. In fact, since the function $\eta(\cdot)$ is decreasing on $[0,1)$  
we get  $\eta(\cdot)\equiv \frac{-\ln\lambda}{\ln\beta}$  on $(.00111,.(01)^\infty)$ and so $\alpha(\eta,t)=+\infty>\eta(t).$
\end{exmp}


\noindent \textbf{Acknowledgments:}
The authors would like to thank Prof. Weisheng Wu  for some useful discussions on this topic.

\end{document}